\documentclass[12pt]{amsart}
\usepackage{graphicx}
\usepackage{amsfonts}
\usepackage{amssymb}
\usepackage{bbm}
\usepackage{times}
\usepackage{amsmath}
\numberwithin{equation}{section}

\allowdisplaybreaks

\theoremstyle{plain}
\newtheorem{theorem}{Theorem}[section]

\newtheorem{lemma}[theorem]{Lemma}
\newtheorem{corollary}[theorem]{Corollary}

\theoremstyle{definition}

\theoremstyle{remark}

\newcounter{other}            

\def\D{\mathbb{D}}

\def\T{\partial \D}
\def\C{\mathbb{C}}
\def\Q{\mathcal{Q}}
\def\qk{\mathcal{Q}_K}
\def \qkt{\qk(\T)}
\def\qp{\Q_p}
\def\qpt{\Q_p(\T)}

\def\f{\frac}
\numberwithin{equation}{section}

\begin{document}

\title[On absolute values of $\qk$ functions]
{On absolute values of $\qk$ functions}

\author[G. Bao]{Guanlong Bao}
\address{Guanlong Bao \\ Department of Mathematics \\ Shantou University \\ Shantou, Guangdong  515063, China}
\email{glbaoah@163.com}

\author[Z. Lou]{Zengjian Lou}
\address{Zengjian Lou \\ Department of Mathematics \\ Shantou University \\ Shantou, Guangdong  515063, China}
\email{zjlou@stu.edu.cn}

\author[R. Qian]{Ruishen Qian}
\address{Ruishen Qian \\ School of Mathematics and Computation Science \\ Lingnan Normal  University \\ Zhanjiang, Guangdong  524048, China}
\email{qianruishen@sina.cn}

\author[H. Wulan]{Hasi Wulan}
\address{Hasi Wulan \\ Department of Mathematics \\ Shantou University \\ Shantou, Guangdong  515063, China}
\email{wulan@stu.edu.cn}
\thanks{The work is supported by NSF of China (No. 11171203, No. 11371234 and No. 11526131),
NSF of Guangdong Province (No. 2014A030313471) and Project of International Science and
Technology Cooperation Innovation Platform in Universities in Guangdong Province (No.
2014KGJHZ007).}


\begin{abstract}
In this paper, the effect of absolute values on the behavior of functions $f$ in the  spaces $\qk$ is investigated. It is clear that $f\in \qkt \Rightarrow |f|\in \qkt$, but the converse is not always true. For $f$ in the Hardy space $H^2$, we give a condition involving the modulus of the function only,  such that this condition together with $|f|\in \qkt$ is  equivalent to $f\in \qk$. As an application, a new criterion for
inner-outer factorisation of $\qk$ spaces is given. These results are also new for $\qp$ spaces.
\end{abstract}

\maketitle

\section{Introduction}

Denote by $\T$ the boundary of  the unit disk $\D$ in the complex plane $\C$.
Let  $H(\D)$  be  the space of functions analytic in $\D$. Throughout this paper, we
assume that $K:[0,\infty)\rightarrow [0, \infty)$ is a
right-continuous and increasing function. A function
$f\in H(\D)$ belongs to the space $\Q_K$ if
$$
\|f\|_{\qk}^2 = \sup_{a\in\D}\,\int_{\D}
|f'(z)|^{2}K\left(g(a, z)\right)dA(z)<\infty,
$$
where $dA$ is the  area measure on $\D$ and $g(a, z)$ is the Green function
in $\D$ with singularity at $a\in\D$. By \cite[Theorem 2.1]{EW}, we know that $\|f\|_{\qk}^2$ is equivalent to
$$
\sup_{a\in\D}\int_\D |f'(z)|^2K\left(1-|\sigma_a(z)|^2\right)dA(z),
$$
where $\sigma_a(z)=\f{a-z}{1-\overline{a}z}$ is a M\"obius transformation of $\D$. If $K(t)=t^p$, $0\leq p<\infty$, then the space $\Q_K$ gives the space
$\Q_p$ (cf. \cite{X1, X2}). In particular,  $\Q_0$ is the Dirichlet space;
$\Q_1=BMOA$, the space of functions with bounded mean oscillation on $\D$;  $\Q_p$ is the Bloch space for all $p>1$.
See  \cite{EW} and \cite{EWX} for more results on $\Q_K$ spaces.
Let $\qkt$ be the space of $f\in L^2(\T)$ with
$$
\|f\|_{\qkt}^2=
\sup_{I\subset \T}\int_I\int_I\frac{|f(\zeta)-f(\eta)|^2}{|\zeta-\eta|^2}K\left(\frac{|\zeta-\eta|}{|I|}\right)|d\zeta||d\eta|<\infty.
$$
Clearly, if $K(t)=t^2$, then $\qkt$ is equal to $BMO(\T)$, the space of functions having bounded mean oscillation on $\T$ (see \cite{G}).

To study $\qk$ and $\qkt$, we usually need two constraints on $K$ as follows.
\begin{equation}\label{Eq 1-1}
\int^1_0\frac{\varphi_K(s)}{s}ds<\infty
\end{equation}
and
\begin{equation}\label{Eq 1-2}
\int^\infty_1\frac{\varphi_K(s)}{s^2}ds<\infty,
\end{equation}
where
$$
\varphi_K(s)=\sup_{0< t\leq 1}K(st)/K(t),\quad 0<s<\infty.
$$

If $K$ satisfies (1.2), then $\qk\subsetneqq BMOA\subsetneqq H^2$,
where $H^2$ denotes the Hardy space in $\D$ (see \cite{D, G}). Thus,
if $K$ satisfies (1.2), then the function $f\in \qk$ has its
non-tangential limit $\widetilde{f}$ almost everywhere on $\T$.
Using the triangle inequality, one gets  that if $g\in \qkt$, then
$|g|$  also belongs to  $\qkt$. In general, the converse is not
true. Consider
$$
g(e^{it})=\begin{cases}
\log t, \ \ & 0<t<\pi, \\
  -\log|t|, \ \  & -\pi<t<0.
   \end{cases}
$$
 By \cite[p. 66]{Gi}, $|g|\in BMO(\T)$, but $g\not\in BMO(\T)$. For $g\in H^2$, it is natural to seek a condition together with $|\widetilde{g}|\in \qkt$
 is  equivalent to $g\in \qk$.
 Our main result, (iii) of Theorem 1.1,  is even new for $\qp$ spaces.

\begin{theorem} \label{th1.1}
Suppose that  $K$ satisfies (1.1) and (1.2). Let $f\in H^2$. Set
$$
d\mu_z(\zeta)=\frac{1-|z|^2}{2\pi|\zeta-z|^2}|d\zeta|,\quad
z\in\D,\quad \zeta\in\partial\D.
$$
Then the following conditions are equivalent.
\begin{enumerate}
\item [(i)] $f\in\qk$.
\item [(ii)] $\widetilde{f}\in\qkt$.
\item[(iii)] $|\widetilde{f}|\in \qkt$ and
\begin{equation}\label{Eq 1-3}
\sup_{a\in
\D}\int_\D\left(\int_{\T}|\widetilde{f}(\zeta)|d\mu_z(\zeta)-|f(z)|\right)^2\frac{K(1-|\sigma_a(z)|^2)}{(1-|z|^2)^2}
dA(z)<\infty.
\end{equation}
\end{enumerate}
\end{theorem}

Applying Theorem 1.1, in Section 4,  we will show a new criterion for inner-outer factorisation of $\qk$ spaces.

In this article, the symbol $A\approx B$ means that $A\lesssim
B\lesssim A$. We say that $A\lesssim B$ if there exists a
constant $C$ such that $A\leq CB$.

\section{Preliminaries}

Given $f\in L^2(\T)$, let $\widehat{f}$ be the Poisson extension of $f$. Namely,
$$
\widehat{f}(z)=\int_{\T} f(\zeta)d\mu_z(\zeta), \ z\in \D.
$$
We first give the following  characterization of $\qkt$ spaces. In particular, if $K(t)=t^p$, $0<p<1$, the corresponding result was proved in \cite{X}.

\begin{theorem} Suppose that  $K$ satisfies  (1.1) and (1.2).  Let $f\in L^2(\T)$. Then $f\in \qkt$ if and only if
\begin{equation}\label{Eq 2-1}
\sup_{a\in\D}\int_\D\left(\int_{\partial\D}|f(\zeta)|^2d\mu_{z}(\zeta)-|\widehat{f}(z)|^2\right)
\frac{K(1-|\sigma_a(z)|^2)}{(1-|z|^2)^2} dA(z)<\infty.
\end{equation}
\end{theorem}
To prove Theorem 2.1, we need the following estimate.

\begin{lemma}
Let (1.1) and (1.2) hold for $K$. If $s<1+c$ \ and \  $2s+r-4\geq 0$, then
$$
 \int_{\D} \frac{K\left(1-|\sigma_a(w)|^2\right)}{(1-|w|^2)^{s}|1-\overline{w}z|^r}dA(w)\approx
\frac{K\left(1-|\sigma_a(z)|^2\right)}{(1-|z|^2)^{s+r-2}}
$$
for all $a, z\in \D$. Here $c$ is a small enough positive constant
which depends only on (1.1) and (1.2).
\end{lemma}
\begin{proof} We point out that $$
 \int_{\D} \frac{K\left(1-|\sigma_a(w)|^2\right)}{(1-|w|^2)^{s}|1-\overline{w}z|^r}dA(w)\lesssim
\frac{K\left(1-|\sigma_a(z)|^2\right)}{(1-|z|^2)^{s+r-2}}
$$
was proved in \cite{BLQW}. So we need only to prove the reverse.
For any $z\in \D$,  let
$$
E(z, 1/2)=\{w\in \D:  |\sigma_z(w)|<1/2\}
$$
be  the pseudo-hyperbolic disk. It is well known that
$$
1-|z|\approx 1-|w|\approx |1-\overline{w}z|
$$
for all $w\in E(z, 1/2)$. Furthermore, by \cite[Lemma 4.30]{Zh}, we have that  $|1-a\overline{w}|\approx|1-a\overline{z}|$ for all $a\in \D$ and $w\in E(z, 1/2)$. Since $K$ satisfies (1.2),   $K(2t)\approx K(t)$ for all $t\in (0, 1)$. We obtain
\begin{eqnarray*}
\int_{\D} \frac{K\left(1-|\sigma_a(w)|^2\right)}{(1-|w|^2)^{s}|1-\overline{w}z|^r}dA(w)
&\geq&\int_{E(z, 1/2)} \frac{K\left(1-|\sigma_a(w)|^2\right)}{(1-|w|^2)^{s}|1-\overline{w}z|^r}dA(w)\\
&\approx&\frac{K\left(1-|\sigma_a(z)|^2\right)}{(1-|z|^2)^{s+r-2}},
\end{eqnarray*}
which gives the desired result.
\end{proof}

\vspace{0.1truecm} \noindent{\bf Proof of Theorem 2.1.} For any
$f\in L^2(\T)$, the Littlewood-Paley identity (\cite[p.
228]{G}) shows that
\begin{equation}\label{Eq 2-2}
\int_\D |\nabla \widehat{f}(w)|^2 \log\f{1}{|w|} dA(w)= \f{1}{2\pi}\int_{\T}|f(\zeta)-\widehat{f}(0)|^2|d\zeta|.
\end{equation}
Replacing $\widehat{f}$ by $\widehat{f\circ\sigma_z}$ in (2.2) for $z\in \D$, one obtains
$$
\int_{\T}|f(\zeta)|^2d\mu_z(\zeta)-|\widehat{f}(z)|^2\approx
\int_\D |\nabla \widehat{f}(w)|^2 (1-|\sigma_z(w)|^2) dA(w).
$$
Using Fubini's theorem and Lemma 2.2, we obtain, for all $a\in \D$, that
\begin{eqnarray*}
&~&\int_\D\left(\int_{\partial\D}|f(\zeta)|^2d\mu_{z}(\zeta)-|\widehat{f}(z)|^2\right)
\frac{K(1-|\sigma_a(z)|^2)}{(1-|z|^2)^2} dA(z)\\
&\approx& \int_\D\left(\int_\D |\nabla \widehat{f}(w)|^2 (1-|\sigma_z(w)|^2) dA(w)\right)
\frac{K(1-|\sigma_a(z)|^2)}{(1-|z|^2)^2} dA(z)\\
&\approx& \int_\D |\nabla \widehat{f}(w)|^2 dA(w)\int_\D \frac{(1-|\sigma_z(w)|^2)K(1-|\sigma_a(z)|^2)}{(1-|z|^2)^2} dA(z)\\
&\approx&\int_\D |\nabla \widehat{f}(w)|^2 K(1-|\sigma_a(w)|^2) dA(w).
\end{eqnarray*}
By \cite{P}, we know that $f\in \qkt$ if and only if
$$
\sup_{a\in \D} \int_\D |\nabla \widehat{f}(z)|^2 K(1-|\sigma_a(z)|)dA(z)<\infty.
$$
Therefore, $f\in \qkt$ if and only if
$$
\sup_{a\in\D}\int_\D\left(\int_{\partial\D}|f(\zeta)|^2d\mu_{z}(\zeta)-|\widehat{f}(z)|^2\right)
\frac{K(1-|\sigma_a(z)|^2)}{(1-|z|^2)^2} dA(z)<\infty.
$$
\hfill{$\square$}

By \cite{EWX}, if (1.1) and (1.2) hold for  $K$  and  $f\in H^2$,
then $f\in \qk$ if and only if $\widetilde{f}\in \qkt$. This
together with  Theorem 2.1, gives the following result immediately
which was also obtained in \cite{WY} by  a  different method.

\begin{corollary}
 Suppose that  $K$ satisfies  (1.1) and (1.2).
Let $f\in H^2$.  Then
 $f\in \qk$ if and only if
$$
\sup_{a\in\D}\int_\D\left(\int_{\partial\D}|\widetilde{f}(\zeta)|^2d\mu_{z}(\zeta)-|f(z)|^2\right)
\frac{K(1-|\sigma_a(z)|^2)}{(1-|z|^2)^2} dA(z)<\infty.
$$
\end{corollary}

\section{Proof of Theorem 1.1}

Recall that $B\in H(\D)$ is called an inner function if $B$ is bounded in $\D$ and $|\widetilde{B}(\zeta)|=1$ for almost every $\zeta\in \T$. An outer function for the Hardy space $H^2$ is the function of the form
$$
O(z)=\eta \exp \left(\int_{\T}\f{\zeta+z}{\zeta-z}\log \psi(\zeta)\f{|d\zeta|}{2\pi}\right), \ \ \eta\in \T,
$$
where $\psi>0$ a.e. on $\T$, $\log \psi\in L^1{(\T)}$ and $\psi \in L^2(\T)$. See \cite{D} for more results on inner and outer functions. Using a technique in \cite{Bo}, we give the proof of Theorem 1.1 as follows.

\vspace{0.1truecm}
\noindent{\bf Proof of Theorem 1.1.}  Note that $(i)\Leftrightarrow (ii)$ was proved in \cite{EWX}.

$(i)\Rightarrow (iii)$. For $f\in \qk$, we have $\widetilde{f}\in \qkt$. The triangle inequality gives that  $|\widetilde{f}|\in \qkt$. For any $z\in \D$,
it follows by H\"older's inequality that
\begin{eqnarray*}
\left(\int_{\partial\D}|\widetilde{f}(\zeta)|d\mu_{z}(\zeta)-|f(z)|\right)^2
&\leq&\left(\int_{\partial\D}|\widetilde{f}(\zeta)-f(z)|d\mu_{z}(\zeta)\right)^2\\
&\leq&\int_{\partial\D}|\widetilde{f}(\zeta)-f(z)|^2d\mu_{z}(\zeta)\\
&=
&\int_{\partial\D}|\widetilde{f}(\zeta)|^2d\mu_{z}(\zeta)-|f(z)|^2.
\end{eqnarray*}
Since $f\in \qk$, the above estimate , together with Corollary 2.3, gives (1.3).

$(iii)\Rightarrow (i)$. If $f\equiv 0$, the result is true. Note that $f\in H^2$.  If $f\not\equiv 0$, then $f$ must be of the form $BO$, where $B$ is an inner function and $O$ is an outer function of $H^2$ (see \cite{D}). By the estimates of $B$ and $O$ respectively, B\"oe \cite[p. 237]{Bo} gave that for any $z\in \D$,
$$
|f'(z)|\leq
\frac{4}{1-|z|}\left(\int_{\T}\left||\widetilde{f}(\zeta)|-\widehat{|\widetilde{f}|}(z)\right|d\mu_z(\zeta)+\widehat{|\widetilde{f}|}(z)-|f(z)|\right).
$$
Here we remind  that
$$
\widehat{|\widetilde{f}|}(z)=\int_{\T}
|\widetilde{f}(\zeta)|d\mu_z(\zeta).
$$
Thus, for any  $a\in \D$, by H\"older's inequality, we deduce that
\begin{eqnarray*}
&~&\int_{\D}|f'(z)|^2K(1-|\sigma_a(z)|)dA(z)\\
&\lesssim&\int_{\D}\left(\int_{\T}\left||\widetilde{f}(\zeta)|-\widehat{|\widetilde{f}|}(z)\right|d\mu_z(\zeta)\right)^2
\frac{K(1-|\sigma_a(z)|)}{(1-|z|^2)^2}dA(z)\\
&~&+\int_{\D}\left(\widehat{|\widetilde{f}|}(z)-|f(z)|\right)^2\frac{K(1-|\sigma_a(z)|)}{(1-|z|^2)^2}dA(z)\\
&\lesssim&\int_{\D}\left(\int_{\T}\left(|\widetilde{f}(\zeta)|-\widehat{|\widetilde{f}|}(z)\right)^2d\mu_z(\zeta)\right)
\frac{K(1-|\sigma_a(z)|)}{(1-|z|^2)^2}dA(z)\\
&~&+\int_{\D}\left(\widehat{|\widetilde{f}|}(z)-|f(z)|\right)^2\frac{K(1-|\sigma_a(z)|)}{(1-|z|^2)^2}dA(z)\\
&\approx&
\int_{\D}\left(\int_{\T}|\widetilde{f}(\zeta)|^2d\mu_z(\zeta)-\left(\widehat{|\widetilde{f}|}(z)\right)^2\right)
\frac{K(1-|\sigma_a(z)|)}{(1-|z|^2)^2}dA(z)\\
&~&+\int_{\D}\left(\widehat{|\widetilde{f}|}(z)-|f(z)|\right)^2\frac{K(1-|\sigma_a(z)|)}{(1-|z|^2)^2}dA(z).
\end{eqnarray*}
By  Theorem 2.1 and (1.3), $f\in\qk$. The proof is complete.
\hfill{$\square$}

\vspace{0.1truecm}

{\bf Remark.} J. Xiao \cite{X} gave an interesting
characterization of $\qp$ spaces in terms of functions with absolute
values. Namely, if $0<p<1$ and $f\in H^2$, then $f\in \qp$ if and
only if $|\widetilde{f}|\in \qpt$ and
$$
\sup_{a\in
\D}\int_\D\left(\left(\int_{\T}|\widetilde{f}(\zeta)|d\mu_z(\zeta)\right)^2-|f(z)|^2\right)
\f{(1-|\sigma_a(z)|^2)^p}{(1-|z|^2)^2}dA(z)<\infty.
$$
Our Theorem 1.1 implies Xiao's result above. In fact, set $K(t)=t^p$, $0<p<1$, in our Theorem 1.1 and Corollary 2.3. Note that
$$
\left(\int_{\T}|\widetilde{f}(\zeta)|d\mu_z(\zeta)\right)^2-|f(z)|^2 \geq \left(\int_{\T}|\widetilde{f}(\zeta)|d\mu_z(\zeta)-|f(z)|\right)^2
$$
and
$$
\left(\int_{\T}|\widetilde{f}(\zeta)|d\mu_z(\zeta)\right)^2 \leq \int_{\T}|\widetilde{f}(\zeta)|^2d\mu_z(\zeta).
$$
Thus, one can obtain Xiao's result directly.

\section{An application to inner-outer factorisation of $\qk$ spaces}

In this section, we will show  a new  criterion for inner-outer decomposition of $\qk$ spaces. In fact, an inner-outer factorisation characterization of $\qk$ spaces has been obtained  in \cite{EWX} as follows.

\vspace{0.3truecm}
\noindent{\bf Theorem A.\  }  {\sl Let $K$ satisfy (1.1) and (1.2) with
$$
\tilde{K}(|z|^2)=-\frac{\partial^2 K(1-|z|^2)}{\partial
z\partial{\overline z}},\quad z\in\D.
$$
Let $f\in H^2$ with $f\not\equiv 0$. Then
 $f\in \qk$ if and only if $f=BO$, where $B$
is an inner function and $O$ is an outer function in $\qk$ for which
\begin{equation}\label{Eq 4-1}
\sup_{a\in\D}\int_\D|O(z)|^2(1-|B(z)|^2)\tilde{K}\left(|\sigma_a(z)|^2\right)|\sigma_a'(z)|^2{dA(z)}<\infty.
\end{equation}}

As an application of Theorem 1.1, we obtain the following result.

\begin{theorem}
Let $K$ satisfy (1.1) and (1.2) with
$$
\tilde{K}(|z|^2)=-\frac{\partial^2 K(1-|z|^2)}{\partial
z\partial{\overline z}},\quad z\in\D.
$$
Let $f\in H^2$ with $f\not\equiv 0$. Then
$f\in \qk$ if and only if $f=BO$, where $B$
is an inner function and $O$ is an outer function in $\qk$ for which
\begin{equation}\label{Eq 4-2}
\sup_{a\in\D}\int_\D|O(z)|^2(1-|B(z)|^2)^2\tilde{K}\left(|\sigma_a(z)|^2\right)|\sigma_a'(z)|^2{dA(z)}<\infty.
\end{equation}
\end{theorem}
{\bf Remark.} Theorem 4.1 shows that formula (4.1) in Theorem A can be replaced by the weaker condition (4.2), and this result is also new for $\qp$ spaces.

\begin{proof}
Necessity.  This is a direct result from Theorem A.

Sufficiency. Let $f=BO$ and $O\in \qk$. Note that  $O\in \qk$  is
equivalent to $\widetilde{O}\in \qkt$. By the triangle inequality, one gets
$|\widetilde{O}|\in \qkt$. Hence $|\widetilde{f}|\in \qkt$. Observe that
\begin{equation}\label{Eq 4-3}
\int_{\T}|\widetilde{f}(\zeta)|d\mu_z(\zeta)-|f(z)|=\int_{\T}|\widetilde{O}(\zeta)|d\mu_z(\zeta)
-|O(z)|+|O(z)|-|B(z)O(z)|.
\end{equation}
Wulan and Ye \cite{WY} gave that  if $K$ satisfies (1.1) and (1.2), then for all $z\in \D$
\begin{equation}\label{Eq 4-4}
\tilde{K}(|z|^2)\approx \frac{K(1-|z|^2)}{(1-|z|^2)^2}.
\end{equation}
By H\"older's inequality, $\widetilde{O}\in \qkt$ and Corollary 2.3, we  show that for any $a\in \D$,
\begin{eqnarray*}
&~&\int_\D\left(\int_{\partial\D}|\widetilde{O}(\zeta)|d\mu_{z}(\zeta)-|O(z)|\right)^2
\frac{K(1-|\sigma_a(z)|^2)}{(1-|z|^2)^2}dA(z)\\
&\leq&\int_\D\left(\int_{\partial\D}\Big|\widetilde{O}(\zeta)-O(z)\Big|^2d\mu_{z}(\zeta)\right)
\frac{K(1-|\sigma_a(z)|^2)}{(1-|z|^2)^2}dA(z)\\
&=&\int_\D\left(\int_{\partial\D}|\widetilde{O}(\zeta)|^2d\mu_{z}(\zeta)-|O(z)|^2\right)
\frac{K(1-|\sigma_a(z)|^2)}{(1-|z|^2)^2}dA(z)<\infty.
\end{eqnarray*}
Combining the above inequality, (4.2),  (4.3) and  (4.4), we get
$$
\sup_{a\in
\D}\int_\D\left(\int_{\T}|\widetilde{f}(\zeta)|d\mu_z(\zeta)-|f(z)|\right)^2\frac{K(1-|\sigma_a(z)|^2)}{(1-|z|^2)^2}
dA(z)<\infty.
$$
Applying Theorem 1.1, we get $f\in \qk$. The proof is complete.
\end{proof}

For $f\in \qk\subseteq H^2$, if we  ignore  the choice of a constant  with modulus one, then $f$ has a unique decomposition with the form
$f(z)=B(z)O(z)$, where $B$ is an inner function and $O$ is an outer function. Combining this  with  Theorem A and Theorem 4.1, we obtain an interesting
result as follows.

\begin{corollary} Suppose that  $K$ satisfies (1.1) and (1.2).
Let    $B$
be an inner function and let $O$ be an outer function in $\qk$. Then the following conditions are equivalent.
\begin{enumerate}
\item [(i)] For some $p\in [1, 2]$,
$$
\sup_{a\in\D}\int_\D|O(z)|^2(1-|B(z)|^2)^p\tilde{K}\left(|\sigma_a(z)|^2\right)|\sigma_a'(z)|^2{dA(z)}<\infty.
$$
\item [(ii)] For all $p\in [1, 2]$,
$$
\sup_{a\in\D}\int_\D|O(z)|^2(1-|B(z)|^2)^p\tilde{K}\left(|\sigma_a(z)|^2\right)|\sigma_a'(z)|^2{dA(z)}<\infty.
$$
\end{enumerate}
\end{corollary}


\begin{thebibliography}{99}

\bibitem{BLQW} G. Bao, Z. Lou, R. Qian and H. Wulan, {\it Improving multipliers and zero sets in $\qk$ spaces}, Collect. Math.  {\bf 66} (2015), no. 3, 453-468.

\bibitem{Bo} B. B\"{o}e, {\it A norm on the holomorphic Besov space},  Proc. Amer. Math. Soc. {\bf 131} (2003), no. 1,  235-241.

\bibitem{Da} K. Dyakonov, {\it Besov spaces and outer functions},  Michigan Math. J. {\bf 45} (1998), no. 1, 143-157.

\bibitem  {D} P. Duren,  {\it Theory of $H^p$ Spaces}, Academic Press, New York, 1970.


\bibitem{EW} M. Ess\'en and H. Wulan, {\it On analytic and meromorphic function and spaces of $\Q_{K}$-type},  Illionis. J. Math. {\bf 46} (2002), no. 4,   1233-1258.

\bibitem{EWX} M. Ess\'en,  H. Wulan and J. Xiao, {\it Several function-theoretic characterizations of M\"obius invariant $\Q_K$ spaces}, J. Funct. Anal. {\bf 230} (2006), no. 1, 78-115.

\bibitem{G} J. Garnett, {\it Bounded Analytic Functions}, Springer, New York, 2007.

\bibitem{Gi} D. Girela, {\it Analytic functions of bounded mean oscillation},  In: Complex Function
Spaces, Mekrij\"arvi 1999  Editor: R. Aulaskari. Univ. Joensuu Dept. Math. Rep. Ser.
4, Univ. Joensuu, Joensuu (2001) pp. 61-170.

\bibitem{P} J. Pau, {\it Bounded M\"obius invariant $\Q_K$ spaces},  J.
Math. Anal. Appl. {\bf 338} (2008), no. 2,  1029-1042.

\bibitem{WY} H. Wulan and F. Ye,  {\it Some results in M\"obius invariant $\qk$ spaces}, Complex Var. Elliptic Equ.  {\bf 60} (2015), no. 11, 1602-1611.

\bibitem {X} J. Xiao, {\it Some results on $\qp$ spaces, $0<p<1$, continued},  Forum Math. {\bf 17} (2005), no. 4, 637-668.

\bibitem  {X1} J. Xiao, {\it Holomorphic $\Q$ Classes}, Springer, LNM 1767, Berlin, 2001.

\bibitem  {X2} J. Xiao, {\it Geometric $\Q_p$ Functions}, Birkh\"auser Verlag, Basel-Boston-Berlin, 2006.

\bibitem  {Zh} K. Zhu,  {\it Operator Theory in Function Spaces}, American Mathematical Society, Providence, RI, 2007.

\end{thebibliography}
\end{document}